\documentclass[12pt,twoside]{amsart}
\usepackage{amsmath,mathrsfs}
\usepackage{amssymb}
\usepackage{amsthm}
\usepackage{amscd}
\usepackage{amsfonts}
\usepackage{graphicx}%
\usepackage{fancyhdr}
\usepackage{verbatim}
\usepackage{color}
\usepackage{hyperref}
\usepackage{xcolor}
\hypersetup{
    colorlinks,
    linkcolor={red!50!black},
    citecolor={blue!50!black},
    urlcolor={blue!80!black}
}
\usepackage{mathrsfs}

\theoremstyle{plain} 
\newtheorem{theorem}{Theorem}

\newtheorem{lemma}[theorem]{Lemma}
\newtheorem{proposition}[theorem]{Proposition}
\theoremstyle{definition}
\newtheorem{definition}[theorem]{Definition}

\newtheorem{remark}[theorem]{Remark}
\newtheorem{example}[theorem]{Example}
\newtheorem*{theorem*}{Theorem}
\newtheorem*{principle*}{Principle}

\usepackage[left=3cm,top=3cm,right=3cm,bottom = 3cm]{geometry}

\numberwithin{theorem}{section}

  \newcommand{\twobytwo}[4]{
\left(\begin{array}{cc}
#1 & #2 \\
#3 & #4 \\	
\end{array}\right)}

\newcommand{\CC}{\mathbb{C}}

\newcommand{\HH}{\mathbb{H}}

\newcommand{\PP}{\mathbb{P}}
\newcommand{\QQ}{\mathbb{Q}}
\newcommand{\RR}{\mathbb{R}}
\newcommand{\ZZ}{\mathbb{Z}}

\newcommand{\cU}{\mathcal{U}}

\DeclareMathOperator{\GL}{GL}

\DeclareMathOperator{\End}{End}
\DeclareMathOperator{\Aut}{Aut}

\DeclareMathOperator{\G}{PSL_2(\ZZ)}
\DeclareMathOperator{\PSL}{PSL}

\DeclareMathOperator{\Jac}{Jac}
\usepackage{caption}

\begin{document}

\title[Jacobians]{Period relations for Riemann surfaces with many automorphisms}

\author[L. Candelori]{Luca Candelori}
\author[J. Fogliasso]{Jack Fogliasso}
\author[C. Marks]{Christopher Marks}
\author[S. Moses]{Skip Moses}

\address{Department of Mathematics, Wayne State University, 656 W Kirby, Detroit, MI 48202, USA}
\email{candelori@wayne.edu}

\address{Department of Mathematics and Statistics, California State University, Chico, 400 West First Street, Chico, CA 95929, USA}
\email{jfogliasso@mail.csuchico.edu;}

\address{Department of Mathematics and Statistics, California State University, Chico, 400 West First Street, Chico, CA 95929, USA}
\email{cmarks@csuchico.edu}

\address{Department of Mathematics and Statistics, California State University, Chico, 400 West First Street, Chico, CA 95929, USA}
\email{smoses7@mail.csuchico.edu}

\begin{abstract}
By employing the theory of vector-valued automorphic forms for non-unitarizable representations, we provide a new bound for the number of linear relations with algebraic coefficients between the periods of an algebraic Riemann surface with many automorphisms. The previous best-known general bound for this number was the genus of the Riemann surface, a result due to Wolfart. Our new bound significantly improves on this estimate, and it can be computed explicitly from the canonical representation of the Riemann surface. As observed by Shiga and Wolfart, this bound may then be used to estimate the dimension of the endomorphism algebra of the Jacobian of the Riemann surface. We demonstrate with a few examples how this improved bound allows one, in some instances, to actually compute the dimension of this endomorphism algebra, and to determine whether the Jacobian has complex multiplication. 
\end{abstract}

\maketitle

\section{Introduction}

The theory of vector-valued modular forms, though nascent in the work of various nineteenth century authors, is a relatively recent development in mathematics. Perhaps the leading motivation for working out a general theory in this area comes from two-dimensional conformal field theory, or more precisely from the theory of vertex operator algebras (VOAs). Indeed, in some sense the article \cite{Zhu} -- in which Zhu proved that the graded dimensions of the simple modules for a rational VOA constitute a weakly holomorphic vector-valued modular function -- created a demand for understanding how such objects work in general, and what may be learned about them by studying the representations according to which they transform. Knopp and Mason \cite{KnoppMason1,KnoppMason2,Mason07} were the first to give a systematic treatment of vector-valued modular forms, with significant contemporary contributions by Bantay and Gannon \cite{BG,Bantay} as well. This initial motivation from theoretical physics leads one naturally to arithmetic considerations, and more recently a number of authors have utilized representations of the modular group to study the {\it bounded denominator conjecture} for noncongruence modular forms \cite{Mason12,FM14,Marks15,FM16,FGM}, the vector-valued version of which originated with Mason as well.

In this article we utilize the vector-valued point of view for another arithmetic application, this time to the classical theory of Riemann surfaces and in particular to the study of the periods of an algebraic Riemann surface. More precisely, suppose $X$ is a compact Riemann surface of genus $g$ defined over $\overline{\QQ}$. For any holomorphic differential $\omega$ defined over $\overline{\QQ}$ and any homology class $[\gamma] \in H_1(X,\ZZ)$, represented by a closed loop $\gamma$ on $X$, the {\em period} associated to $\omega$ and $\gamma$ is the complex number $\int_{\gamma} \omega$. Following \cite{Wolfart}, let  
$$
V_X :=  \mathrm{span}_{\overline{\QQ}} \left\{ \int_{\gamma} \omega : \gamma \in H_1(X,\ZZ),\ \omega \text{ defined over } \overline{\QQ} \right\}, 
$$
viewed as a $\overline{\QQ}$-subspace of $\CC$. Since the dimension of the space of $\overline{\QQ}$-differentials is $g$ and the rank of $H_1(X,\ZZ)$ is $2g$, we deduce that $\dim_{\overline{\QQ}} V_X \leq 2g^2$. This bound is saturated, and in fact it is realized by generic compact Riemann surfaces of genus $g$. 

There are, however, special classes of algebraic curves where the above bound is a substantial overestimate. For example, in \cite{Wolfart} Wolfart studies Riemann surfaces with {\em many automorphisms}. These curves are characterized by the fact that if one fixes any desired genus and automorphism group, the resulting family of curves is zero-dimensional. For such a curve $X$, it is proven in \cite{Wolfart} that 
\begin{equation}
\label{eq:trivialBoundIntro}
\dim_{\overline{\QQ}} V_X \leq g,
\end{equation}
reflecting the fact that the larger automorphism group imposes additional linear relations among the periods. 

In this article we improve the bound \eqref{eq:trivialBoundIntro}, by inserting terms depending on the {\em canonical representation} of $X$. More precisely, let $X$ be a curve with many automorphisms, let $G$ be the group of automorphisms of $X$, and denote by 
$$
\rho_X: G^{\rm op} \longrightarrow \GL(\Omega^1(X))
$$
be the $g$-dimensional complex representation given by the action of $g\in G$ on differentials given by pull-back (note that pull-back is contravariant, thus the `op' is needed to obtain a group representation). This is called the {\em canonical representation} of $X$, since $\Omega^1(X)$ is also known as the {\em canonical bundle} of $X$. Its decomposition into irreducible characters can readily be computed using the {\em Chevalley-Weil} formula \cite{Chevalley-Weil} (see \cite{Candelori-CW} for a modern account). Now any Riemann surface $X$ with many automorphisms can actually be uniformized as a quotient \cite{Wolfart} 
$$
X(N):= \mathbb{H}/N , \quad N\triangleleft \Delta
$$
where $\mathbb{H}$ denotes the complex upper half-plane and $N\triangleleft \Delta$ is a finite-index normal subgroup of a {\em co-compact Fuchsian triangle group} $$
\Delta = \Delta(p,q,r) = \{ \delta_0, \delta_1, \delta_{\infty} : \delta_0^p = \delta_1^q = \delta_{\infty}^r = \delta_0\delta_1\delta_{\infty}\}
$$
acting as linear fractional transformations on $\mathbb{H}$, so that $G$ appears as the finite quotient $\Delta/N$. In this presentation there are now three generators $\delta_0, \delta_1, \delta_{\infty}$ of $G$, which, by abuse of notation, represent the images of the corresponding generators of $\Delta$. In terms of these generators, we prove:
\begin{theorem}
\label{theorem:MainThmIntro}
Let $X$ be a curve with many automorphisms, uniformized as $X = X(N)$ and with canonical representation $\rho_N := \rho_{X(N)}$. Then 
$$
\dim_{\overline{\QQ}} V_X \leq g - d_0 - d_1 - d_{\infty}
$$
where $d_x = \dim_{\CC}(\ker(\rho_N(\delta_x) - I_g))$ is the geometric multiplicity of the eigenvalue 1  in $\rho_N(\delta_x)$, $x\in \{0,1,\infty\}$.  
\end{theorem}

In practice, the numbers $d_0,d_1,d_{\infty}$ are easy to compute given the character of $\rho_N$, and they tend to be linear in $g$, providing a significant improvement over the bound \eqref{eq:trivialBoundIntro}. We demonstrate sample computations of the bound of Theorem \ref{theorem:MainThmIntro} in Section \ref{sec:Examples} below. 

The proof of Theorem \ref{theorem:MainThmIntro} builds on the ideas of \cite{Candelori-Marks}, and it relies crucially on the theory of vector-valued modular forms for non-unitarizable representations, as pioneered by Knopp and Mason \cite{KnoppMason1}, \cite{KnoppMason2}. The proof of the Theorem, together with the connection with modular forms, can be found in Sections  \ref{sec:vvmfAndPeriods} and \ref{sub:sharperBound} below.  

As demonstrated already by Shiga and Wolfart in \cite{Shiga-Wolfart}, the knowledge of the  integer $\dim_{\overline{\QQ}} V_X$ is essentially equivalent to the knowledge of the dimension of the $\mathbb{Q}$-algebra of endomorphisms of $\Jac(X)$ (see Theorem \ref{thm:ShigaWolfartTheorem} below for a precise statement). There are many open questions about the endomorphism algebras of Jacobians, such as Coleman's conjecture regarding Riemann surfaces with {\em complex multiplication} and the Ekedahl-Serre question regarding Riemann surfaces whose Jacobian is isogenous to a product of elliptic curves. These questions may be viewed as the main motivation for studying linear relations between the periods of $X$. We demonstrate in Section \ref{sec:Examples} how Theorem \ref{theorem:MainThmIntro}, in conjuction with the theorems of  Shiga-Wolfart, allow one to compute the decomposition of the Jacobians of some Riemann surfaces with many automorphisms. 

{\bf Acknowledgments.} The first author would like to thank Bill Hoffman and Ling Long at Louisiana State University for helpful observations during the early stages of this project, as well as Tony Shaska at Oakland University of helpful conversations. The third author is extremely happy to acknowledge the ongoing contribution of Geoffrey Mason to his research career, and this article is dedicated to him on the occasion of his 70th birthday. The second, third, and fourth authors all benefited from internal grants funded by the Research Foundation and the College of Natural Sciences at California State University, Chico.

\section{Riemann surfaces with many automorphisms}
\label{sub:RSWithAutos}
Suppose $X$ is a compact, connected Riemann surface. By the uniformization theorem of Riemann surfaces, $X$ is isomorphic to a quotient $\Gamma\backslash\cU$, where $\cU$ is either the complex upper half-plane $\HH$, the Riemann sphere $\PP^1$, or the complex plane $\CC$ and $\Gamma\subseteq \Aut(\cU)$ is a discrete subgroup. By Riemann's existence theorem, $X$ also possesses the structure of an algebraic curve over $\CC$: that is, it can be defined as the zero locus of a collection of polynomials with complex coefficients inside a suitable projective space. Suppose that these polynomials can be chosen to all have coefficients in the field of algebraic numbers $\bar{\QQ}$. Then we say that $X$ is {\em defined} over $\bar{\QQ}$. In this case, we can be more specific about the group $\Gamma$ uniformizing $X$:

\begin{theorem}[Belyi \cite{Belyi}, Wolfart \cite{Wolfart}]
\label{thm:BelyiWolfart}
Let $X$ be a compact, connected Riemann surface defined over $\bar{\QQ}$. Then $X$ is isomorphic to $\Gamma\backslash\HH$, for some finite-index subgroup $\Gamma$ of a co-compact Fuchsian triangle group $\Delta\subseteq \PSL_2(\RR)\simeq \Aut(\HH)$.  
\end{theorem}

For any choice of positive integers $p,q,r > 0$, the {\em triangle group} $\Delta = \Delta(p,q,r)$ is defined abstractly as the infinite group generated by three elements $\delta_0, \delta_1, \delta_{\infty}$ with presentation
$$
\Delta = \Delta(p,q,r) = \{ \delta_0, \delta_1, \delta_{\infty} : \delta_0^p = \delta_1^q = \delta_{\infty}^r = \delta_0\delta_1\delta_{\infty}\}.
$$
The triangle group is {\em Fuchsian} if it can be embedded in $\PSL_2(\RR)$, and it is {\em co-compact} if the corresponding quotient $\Delta\backslash\HH$ is compact. In this article $\Delta$ will always be Fuchsian, but not necessarily co-compact. For example, it is customary to allow any of $p,q,r$ to be $\infty$ whenever the corresponding generator is of infinite order, so that, for example, 
$$
\Delta(\infty, \infty, \infty) = F_2 = \Gamma(2), \quad \Delta(2, 3, \infty) = \PSL_2(\ZZ). 
$$
When this happens the triangle groups are no longer co-compact, in which case we apply a suitable compactification to the quotient $\Delta\backslash\HH$ by adding {\em cusps} (as is the case for {\em modular curves}, for example). In either case, we let $X(\Delta)$ denote the corresponding Riemann surface, and similarly we denote by $X(\Gamma)$ the compact Riemann surface corresponding to finite-index subgroups $\Gamma\leq\Delta$, each of which yields a finite ramified cover $X(\Gamma)\rightarrow X(\Delta)$.

Theorem \ref{thm:BelyiWolfart} is remarkable in that it allows one to study algebraic Riemann surfaces from the point of view of group theory and the representation theory of the triangle groups $\Delta(p,q,r)$, which are very much amenable to computations. For example, given a fixed triangle group $\Delta$, we get a 1-1 correspondence
$$
\left\{\text{finite-index normal subgroups } N \triangleleft \Delta \right\} \longleftrightarrow \left\{\text{ramified Galois covers } X(N)\rightarrow X(\Delta)\right\}. 
$$
As is the case in algebraic number theory, where questions about general finite extensions $L/K$ of fields can be tackled by first studying {\em Galois} extensions, here too the study of an algebraic Riemann surface $X(\Gamma)\rightarrow X(\Delta)$ can often be reduced to studying {\em Galois} covers, by taking a suitable `Galois closure'. By the above correspondence, that is the same as restricting to {\em normal} subgroups $N\subseteq \Delta$. For simplicity we will do so in this article, keeping in mind that all the results presented can be extended to the case of a general Riemann surface, not necessarily uniformized by a normal subgroup.

Now, when $N\triangleleft\,\Delta$ is normal, the finite quotient $G = \Delta/N$ acts as a group of automorphisms on $X(N)$, which partially justifies the following:

\begin{definition}[Wolfart \cite{Wolfart}]
\label{def:ManyAutos}
A Riemann surface $X$ defined over $\bar{\QQ}$ has {\em many automorphisms} if there is an isomorphism $X \simeq X(N)$, where $N\triangleleft\,\Delta$ is a normal subgroup of a triangle group $\Delta$. 
\end{definition}

In general, let $A(G)$ be the locus of genus $g$ Riemann surfaces with automorphism group equal to $G$, viewed as a subvariety of the moduli space of all Riemann surfaces of genus $g$ (see, e.g., \cite{Magaard-Shaska}). Then $X$ has many automorphisms if and only if $\dim A(G) = 0$.

\section{Endomorphism algebras of abelian varieties and transcendence}
\label{sub:EndoAlgebras}

Suppose now that $A$ is an abelian variety of dimension $g>0$ defined over $\bar{\QQ}$. Equivalently, $A$ is a complex torus $\CC^g/\Lambda$ together with an embedding into projective space defined by polynomial equations with coefficients in $\bar{\QQ}$. In this case, the $g$-dimensional $\CC$-vector space of holomorphic 1-forms for $A$ contains the $\bar{\QQ}$-vector space of 1-forms that are {\em defined over $\bar{\QQ}$}, that is, those differentials that come by base-change from the regular, K\"{a}hler differentials on the underlying algebraic variety. These $\bar{\QQ}$-differentials can be integrated over homology classes of cycles $\gamma \in H^1(A,\ZZ)$, and the resulting complex number 
$$
\int_{\gamma} \omega \in \CC
$$ 
is a {\em period} of $\omega$. This integration process is only possible by viewing $A$ as a complex manifold, a highly transcendental operation, and therefore periods tend to be transcendental numbers even if $\omega$ is defined over $\bar{\QQ}$. To keep track of `how many' transcendental periods we get, we make the following important definition:

\begin{definition}[\cite{Wolfart}]
Let $A$ be an abelian variety defined over $\bar{\QQ}$. The {\em period $\bar{\QQ}$-span} of $A$ is the $\bar{\QQ}$-vector space
$$
V_A := \mathrm{span}_{\bar{\QQ}}\left\{ \int_{\gamma} \omega : \gamma \in H^1(A,\ZZ),\ \omega \text{ defined over } \bar{\QQ}\right\} \subseteq \CC
$$
\end{definition}

If we regard $\CC$ as an infinite-dimensional vector space over $\bar{\QQ}$, then $V_A$ is seen to be a $\bar{\QQ}$-subspace of $\CC$. It is clearly finite-dimensional: indeed, the $\bar{\QQ}$-space of algebraic differential 1-forms for $A$ is $g$-dimensional, by base-change, and the rank of $H^1(A,\ZZ)$ as a free $\ZZ$-module is $2g$. Therefore 
$$
\dim_{\bar{\QQ}} V_A \leq 2g^2
$$
for any abelian variety $A$ defined over $\bar{\QQ}$. We call this the {\em trivial bound} on $\dim_{\bar{\QQ}} V_A$. 

How can we get `extra' linear $\bar{\QQ}$-relations on $V_A$? If $\phi$ is a non-scalar  endomorphism of $A$ (defined over $\bar{\QQ}$) then it is easy to show that $\phi$ induces a $\bar{\QQ}$-relation via its action by pull-back on $\omega$ and its natural action on $H^1(A,\ZZ)$. The amazing fact is that {\em all} such linear relations are given in this way \cite{Shiga-Wolfart}, as follows from W\"{u}stholz' analytic subgroup theorem.  To state the precise result, let 
$$
\End_0(A):= \End(A)\otimes\mathbb{Q}
$$
be the endomorphism algebra of $A$. 

\begin{theorem}[Shiga-Wolfart, \cite{Shiga-Wolfart}]
\label{thm:ShigaWolfartTheorem}
Suppose $A$ is a simple abelian variety of dimension $g$. Then 
$$
\dim_{\bar{\QQ}} V_A = \frac{2g^2}{\dim_{\QQ}\End_0(A)}.
$$
More generally, if $A$ is isogenous to a product $A_1^{k_1} \times \cdots \times A_m^{k_m}$ of simple abelian varities, each of dimension $g_i$, then 
$$
\dim_{\bar{\QQ}} V_A = \sum_{i=1}^m\frac{2g_i^2}{\dim_{\QQ}\End_0(A_i)}.
$$
\end{theorem}

Therefore, the knowledge of $\dim_{\bar{\QQ}} V_A$ gives information about the dimension of the endomorphism algebra of $A$, and it is often enough to actually determine $\dim_{\QQ}\End_0(A)$. 

\begin{example}
\label{ex:CMEcurveSchneider}
Suppose $g=1$, so that $A=E$ is an elliptic curve over $\bar{\QQ}$. In this case the space of $\bar{\QQ}$-differentials is one-dimensional, say, generated by $\omega$, and the integral homology $H^1(E,\ZZ)$ is a rank 2 free $\ZZ$-module generated by, say, $\gamma_1$ and $\gamma_2$. The periods of $\omega$ are the nonzero complex numbers
$$
\Omega_1 := \int_{\gamma_1} \omega, \quad \Omega_2 := \int_{\gamma_2} \omega. 
$$ 
Consider the ratio $\tau := \Omega_2/\Omega_1$. Generically, this ratio will be transcendental, so that $\dim_{\bar{\QQ}} V_E = 2$ and the elliptic curve attains the trivial bound of $2g^2$. By Theorem \ref{thm:ShigaWolfartTheorem}, this forces $\End_0(E) = \mathbb{Q}$, that is, the only endomorphisms of $E$ act as scalars.  On the other hand, when $\tau\in \bar{\QQ}$, then Theorem \ref{thm:ShigaWolfartTheorem} implies that $\dim_{\mathbb{Q}}\End_0(E) = 2$. As is well-known \cite{Silverman}, in this case $\End(E)$ is an order inside the ring of integers of an imaginary quadratic extension of $\mathbb{Q}$, and the elliptic curve is said to have {\em complex multiplication}. We thus recover the famous theorem of Schneider from 1936: 
$$
\tau \in \bar{\QQ} \Longleftrightarrow E \text{ has complex multiplication}.
$$

\end{example}

\begin{example}
\label{ex:trivialEndo}
In general, when $\dim_{\bar{\QQ}} V_A = 2g^2$ then Theorem \ref{thm:ShigaWolfartTheorem} implies  that $\End(A) = \ZZ$. 
\end{example}

\section{Jacobians with many automorphisms}
\label{sub:JacobiansManyAutos}

Suppose now that $X$ is a Riemann surface of genus $g>0$ and let $J(X)$ be its {\em Jacobian variety}, the abelian variety of dimension $g$ canonically associated to $X$ by taking the torus $\CC^g/\Lambda$, where $\Lambda$ is the rank $2g$ lattice spanned by the periods of $X$. If $X = X(N)\rightarrow X(\Delta)$ has many automorphisms (Def. \ref{def:ManyAutos}), then $X$ is defined over $\bar{\QQ}$, and its Jacobian variety $J(X)$ is an abelian variety also defined over $\bar{\QQ}$ \cite{Milne}. In this case, by using the fact that the group of automorphisms of the Galois cover $X(N)\rightarrow X(\Delta)$ acts transitively on $H_1(X(N),\ZZ)$,  it is possible to give a sharper general bound on the dimension of the period $\bar{\QQ}$-span of $\Jac(X)$: 

\begin{theorem}[Wolfart \cite{Wolfart}]
\label{thm: WolfartBound}
Suppose $X$ is a Riemann surface of genus $g>0$ with many automorphisms, let $J(X)$ be its Jacobian variety and let $V_X := V_{\Jac(X)}$ be the period $\bar{\QQ}$-span of $\Jac(X)$. Then 
$$
\dim_{\bar{\QQ}} V_X  \leq g.
$$
\end{theorem}

This bound on $\dim_{\bar{\QQ}} V_X $ is so far the only known for a general Riemann surface with many automorphisms. Though weak, it can already be used to deduce special structures within the endomorphism algebras of Jacobians with  many automorphisms, as in the two examples below.

\begin{example}
\label{ex:genus1manyAutos}
Suppose $X$ has genus $g=1$, and that $X$ has many automorphisms. Then $\dim_{\bar{\QQ}} V_X = 1$ so that $\Jac(X) \simeq X$ has complex multiplication (compare to Example \ref{ex:CMEcurveSchneider} in Section \ref{sub:EndoAlgebras}). 
\end{example}

\begin{example}
In general, if $X$ has many automorphisms then $\End(\Jac(X)) \neq \ZZ$ always (compare to Example \ref{ex:trivialEndo} in Section \ref{sub:EndoAlgebras}).  
\end{example}

\section{Vector-valued modular forms and periods}
\label{sec:vvmfAndPeriods}
An excellent reference for the Riemann surface and automorphic function theory described in this section is Shimura's classic text \cite{Shimura}, and the vector-valued point of view we utilize here emerges naturally out of work of Knopp and Mason \cite{KnoppMason1,KnoppMason2,Mason07,Mason08} . We also note that the construction of Theorem \ref{thm:MainVVmfThm} below is a generalization of an idea used in Theorem 1.4 of the recent article [CHMY], where such a result was proven in a different manner for subgroups of the modular group $\G$.

Suppose $N$ is a normal, finite index subgroup of a Fuchsian group of the first kind $\Gamma\leq$ PSL$_2(\mathbb{R})$, and let $\mathbb{H}^*$ denote the union of $\mathbb{H}$ with the cusps of $N$ in $\mathbb{R}\cup\{\infty\}$ (note that if $\Gamma$ is cocompact in PSL$_2(\mathbb{R})$ then so is $N$, and this set of cusps is empty). Let $X(N)=N\backslash\mathbb{H}^*$ denote the compact Riemann surface $X(N)$ associated to $N$, and assume that the genus of $X(N)$ is $g\geq1$. The finite group $G= \Gamma/N$ acts as  automorphisms of the cover $X(N)\rightarrow X(\Gamma)$ and, via pullback (which is contravariant), acts on the vector space of differential forms for $X(N)$. This defines a $g$-dimensional linear representation  $G^{\rm{op}}\rightarrow \GL_g(\CC)$, the {\em canonical representation} of the cover. We may lift this representation to a group representation
\begin{equation}
\label{eq:liftedCanonicalRep}
\rho_N: \Gamma^{\rm{op}} \rightarrow \GL_g(\CC)
\end{equation}
via the quotient map $\Gamma \rightarrow \Gamma/N$. Let $\{\omega_1, \ldots, \omega_g\}$ be a basis for the space $\Omega^1(X(N))$ of holomorphic 1-forms for $X(N)$. Under the uniformization isomorphism $X(N) \simeq N\backslash\mathbb{H}^*$, we can lift $\omega_1, \ldots, \omega_g$ to a basis $\{f_1, \ldots, f_g\}$ of weight two cusp forms on $N$. The vector
\[F:=(f_1,\cdots,f_g)^t\]
is then a weight two vector-valued cusp form for the opposite representation $\rho:=\rho^{\rm opp}_{N}$. In other words, $F(\tau)$ is holomorphic in $\mathbb{H}$, vanishes at any cusps that $N$ may have, and for each $\gamma=\twobytwo{a}{b}{c}{d}\in\Gamma$ the functional equation
\begin{equation}\label{eq:vvcuspform}
F(\gamma\tau)(c\tau+d)^{-2}=\rho(\gamma)F(\tau)
\end{equation}
is satisfied.  

Fix a base-point $\tau_0\in X(N)$. For each $1\leq k\leq g,$ let 
\[u_k(\tau)=\int_{\tau_0}^\tau f_k(z)\,dz.\]
These functions have an important transformation property:

\begin{theorem}
\label{thm:MainVVmfThm}
$U=(u_1,\ldots,u_g,1)^t$ is a holomorphic vector-valued automorphic function for a representation $\pi_N: \Gamma\rightarrow \GL_{g+1}(\CC)$ arising from a  non-trivial extension of the form 
\[0 \rightarrow \rho \rightarrow \pi_{N} \rightarrow 1 \rightarrow 0.\]
\end{theorem} 

\begin{proof} Writing $U(\tau)=\left(\int_{\tau_0}^\tau F(z)\,dz,1\right)^t$, we first note that the functional equation (\ref{eq:vvcuspform}) for $F$ may be written as 
\[F(\gamma\tau)\,d(\gamma\tau)=\rho(\gamma)F(\tau)\,d\tau.\]
For each $\gamma\in\G$, set
\begin{equation}\label{eq:periods}
\Omega(\gamma)=\int_{\tau_0}^{\gamma\tau_0} F(z)\,dz.
\end{equation}
Then for any $\gamma\in\Gamma$ we have
\begin{align*}U(\gamma\tau)&=\begin{pmatrix}\int_{\tau_0}^{\gamma\tau}F(z)\,dz\\1\end{pmatrix}\\
&=\begin{pmatrix}\int_{\tau_0}^{\gamma\tau_0}F(z)\,dz+\int_{\gamma\tau_0}^{\gamma\tau}F(z)\,dz\\1\end{pmatrix}\\
&=\begin{pmatrix}\Omega(\gamma)+\int_{\tau_0}^\tau F(\gamma z)\,d(\gamma z)\\1\end{pmatrix}\\
&=\begin{pmatrix}\Omega(\gamma)+\rho(\gamma)\int_{\tau_0}^\tau F(z)\,dz\\1\end{pmatrix}\\
&=\pi_N(\gamma)U(\tau)
\end{align*}
where we define 
\[\pi_N(\gamma)=\twobytwo{\rho(\gamma)}{\Omega(\gamma)}{0}{1}\in\GL_{g+1}(\CC).\]
One sees now that $\pi_N$ is a representation of $\G$ if and only if the relation
\[\Omega(\gamma_1\gamma_2)=\rho(\gamma_1)\Omega(\gamma_2)+\Omega(\gamma_1)\]
obtains for any $\gamma_1,\gamma_2\in\G$. Similar to the above computation, we have
\begin{align*}
\Omega(\gamma_1\gamma_2)&=\int_{\tau_0}^{(\gamma_1\gamma_2)\tau_0}F(z)\,dz\\
&=\int_{\tau_0}^{\gamma_1\tau_0}F(z)\,dz+\int_{\gamma_1\tau_0}^{\gamma_1(\gamma_2\tau_0)}F(z)\,dz\\
&=\Omega(\gamma_1)+\int_{\tau_0}^{\gamma_2\tau_0}F(\gamma_1z)\,d(\gamma_1z)\\
&=\Omega(\gamma_1)+\rho(\gamma_1)\int_{\tau_0}^{\gamma_2\tau_0}F(z)\,dz\\
&=\Omega(\gamma_1)+\rho(\gamma_1)\Omega(\gamma_2)
\end{align*}
so $\pi_N$ is indeed a representation. To see that this extension of $\rho$ by the trivial character of $\Gamma$ does not split, we simply need to observe that if it did, then $\ker(\pi_N)$ would contain $N$. This would imply that the Riemann surface $X(N)$ admits nonconstant holomorphic functions, namely the $u_k$, but this is impossible since $X(N)$ is compact.  This completes the proof of the theorem. 
\end{proof}

For each $\gamma \in N$, the map $\tau_0 \mapsto \gamma\tau_0$ defines a closed loop on $X(N)$. One observes that the proof of Proposition 1.4 in \cite{Manin} remains valid for any Fuchsian group of the first kind in PSL$_2(\RR)$, so there is a surjective homomorphism of groups $N\rightarrow H_1(X(N),\ZZ)$ given by $\gamma \rightarrow [\tau_0 \mapsto \gamma\tau_0]$ whose kernel is generated by the commutator subgroup $N'$ together with the elliptic and parabolic elements of $N$. Thus there is some set $\gamma_j$, $1\leq j\leq2g$, of hyperbolic elements of $N$ that surjects onto a basis of $H_1(X(N),\ZZ)$, and using the above notation we find that the vectors $\Omega(\gamma_j)\in\CC^g$ generate the period lattice $\Lambda$ that defines the Jacobian of $X(N)$. On the other hand, from the proof of Theorem \ref{thm:MainVVmfThm} we see for each $j$ we have 
\[\pi_N(\gamma)=\twobytwo{\rho(\gamma)}{\Omega(\gamma)}{0}{1},\]
so in fact these vectors appear in the matrices defining the extension representation $\pi_N$. Thus, when these periods are algebraic Theorem \ref{thm:MainVVmfThm} allows one to compute them explicitly, as was already demonstrated for modular curves in \cite{Candelori-Marks}. Even when they cannot be computed explicitly (since in general at least some of them are transcendental), the existence and analysis of $\pi_N$ yields important information about the periods, as we demonstrate below in Section \ref{sub:sharperBound}.

\section{Group cohomology}
\label{section:groupCohomology}

Let $N$ be a normal subgroup of finite index in a Fuchsian group $\Gamma$ of the first kind, and let $\rho_N$ be the canonical representation associated to the cover $X(N)\rightarrow X(\Gamma)$. Theorem \ref{thm:MainVVmfThm} shows that the periods of the curve $X(N)$ are the matrix entries of a non-trivial extension $\pi_N$ of the trivial representation by $\rho_N$. In this section we briefly recall how such extensions are classified by {\em group cohomology}. 

Let $\rho:\Gamma \rightarrow \GL(V)$ be a finite-dimensional complex representation of $\Gamma$. For any integer $n\geq 1$, denote by $C^n(\Gamma, \rho)$ the group of functions $\Gamma^n \rightarrow V$ (the {\em n-cochains}) and let $d^{n+1}: C^n\rightarrow C^{n+1}$ be the {\em coboundary} homomorphism given by
\begin{align*}
d^{n+1}\kappa(g_1,\ldots, g_{n+1}) = \rho(g_1)\kappa(g_2,&\ldots,g_n) + \sum_{i=1}^n (-1)^i \kappa(g_1, \ldots, g_{i-1}, g_ig_{i+1}, \ldots, g_{n+1}) + \\
 &+ (-1)^{n+1}\kappa(g_1, \ldots, g_n)
\end{align*}
For $n\geq 0$ let $Z^n(\Gamma,\rho):= \ker d^{n+1}$ be the group of {\em n-cocyles} and for $n\geq 1$ let $B^n(\Gamma,\rho):= \mathrm{ im}\, d^{n}$ be the group of {\em n-coboundaries}. Since $d^{n+1}\circ d^n = 0$, we have $B^n \subset Z^n$ for all $n\geq 1$ (so $(C^{\bullet},d^{\bullet})$ forms a {\em cochain complex}) and we can define the {\em n-th group cohomology} by 
\begin{align*}
H^0(\Gamma,\rho) &:= V^{\rho} = \{ v\in V: \rho(v) = v\} \\
H^n(\Gamma,\rho) &:= \frac{Z^n(\Gamma,\rho)}{B^n(\Gamma,\rho)}, \quad n\geq 1.
\end{align*}
In particular, for $n=1$ we have:
\begin{align*}
Z^1(\Gamma,\rho) &= \{ \kappa: \Gamma\rightarrow V \,|\,  \kappa(g_1g_2) = \rho(g_1)\kappa(g_2) + \kappa(g_1) \} \\
B^1(\Gamma,\rho) &= \{ \kappa: \Gamma\rightarrow V \,|\, \kappa(g) = \rho(g)v - v \text{ for some } v \in V\}.
\end{align*}

Let now $\mathrm{Ext}_{\Gamma}^1(\rho,1)$ be the set of isomorphism classes of extensions of $\Gamma$-representations $\rho \rightarrow \pi \rightarrow 1$, where an isomorphism $\pi  \simeq \pi'$ is an isomorphism of $\Gamma$-representations which restricts to an isomorphism $\rho \simeq \rho'$. Given any $\pi \in \mathrm{Ext}_{\Gamma}^1(\rho,1)$, we can find a basis for $\pi$ such that 
\[\pi(g) = \twobytwo{\rho(g)}{\kappa_{\pi}(g)}{0}{1},\]
as in the proof of Theorem \ref{thm:MainVVmfThm}. An easy computation shows that $\kappa_{\pi} \in Z^1(\Gamma,\rho)$ and that its class in $H^1$ completely classifies the extension $\pi$: 
\begin{proposition}
\label{prop:ExtensionsGroupCohomology}
The map $\pi \mapsto \kappa_{\pi}$ gives a bijection
\[\mathrm{Ext}_{\Gamma}^1(\rho,1) \longleftrightarrow H^1(\Gamma, \rho).\]\qed
\end{proposition}

Note that the group structure on $H^1(\Gamma, \rho)$ goes over naturally to give the {\em Baer sum} of extensions on $\mathrm{Ext}_{\Gamma}^1(\rho,1)$, and therefore the bijection of Proposition \ref{prop:ExtensionsGroupCohomology} can be promoted to a group isomorphism, and even to an isomorphism of $\CC$-vector spaces. The zero vector in $\mathrm{Ext}_{\Gamma}^1(\rho,1)$, corresponding to 1-coboundaries in $B^1(\Gamma,\rho)$, is given by those representations $\pi$ which are completely reducible, $\pi \simeq \rho\oplus 1$. In particular, note that if $\Gamma$ is a finite group then $H^1(\Gamma,\rho) = 0$ always, since in this case the category of $\Gamma$-representations is semi-simple. Going back to Theorem \ref{thm:MainVVmfThm}, for example, we see that the canonical representation $\rho_N$ cannot have any non-trivial extensions $\pi_N$ of it by 1 if viewed as a representation of $G:= \Gamma/N$. However, when lifted to a representation of $\Gamma$, Theorem \ref{thm:MainVVmfThm} says that a {\em canonical} non-trivial extension class 
$$
\pi_N \in H^1(\Gamma,\rho_N)
$$
becomes available, and is determined by the periods of $X(N)$. In the next section we study the properties of this class and connect it to the endomorphism algebra of the Jacobian of $X(N)$.

\section{A sharper bound on the dimension of the period $\bar{\QQ}$-span of Riemann surfaces with many automorphisms}
\label{sub:sharperBound}

We now go back to the case of triangle groups $\Delta = \Delta(p,q,r)$. Let $N\triangleleft\,\Delta$ and let $\rho_N$ be the canonical representation of $X(N)\rightarrow X(\Delta)$, lifted to a representation of $\Delta$ as in \eqref{eq:liftedCanonicalRep}. As explained in Section \ref{section:groupCohomology}, Theorem  \ref{thm:MainVVmfThm} gives a canonical class  $\pi_N \neq 0 \in H^1(\Delta, \rho_N)$ encoding the periods of $X(N)$. Using this new invariant of the curve $X(N)$, we obtain the following bound:

\begin{theorem}
\label{thm:MainThm}
Suppose $X(N)\rightarrow X(\Delta)$ is a Riemann surface of genus $g>0$ with many automorphisms, and let $\rho_N: \Delta^{\rm{op}} \rightarrow \GL_g(\CC)$ be its canonical representation. Let $V_{X(N)} := V_{\Jac(X(N))}$ be the span of $\bar{\QQ}$-periods  of $\Jac(X(N))$. Then 
$$
\dim_{\bar{\QQ}} V_{X(N)} \leq g - d_0 - d_1 - d_{\infty}
$$
where $d_x = \dim_{\CC}(\ker(\rho_N(\delta_x) - I_g))$ is the geometric multiplicity of the eigenvalue 1  in $\rho_N(\delta_x)$, $x\in \{0,1,\infty\}$. 
\end{theorem}

\begin{proof}
First, let 
$$
\tilde{\Delta}:= \Delta(p,q,\infty) = \{\delta_0, \delta_1 : \delta_0^{p} = 1, \delta_1^q = 1\} \simeq \ZZ/p\ZZ * \ZZ/q\ZZ. 
$$
Let $\delta_{\infty}:= \delta_1^{-1}\delta_0^{-1} \in \tilde{\Delta}$ and let $K$ denote the smallest normal subgroup of $\tilde{\Delta}$ containing $\delta_{\infty}^r$. Taking the quotient by $K$ determines a surjective homomorphism $\tilde{\Delta} \rightarrow \Delta$, and we can lift $\rho_N$ to a representation of $\tilde{\Delta}$ (we will abuse notation and denote both $\rho_N$ and its lift to $\tilde{\Delta}$ by the same symbol). We need the following lemma about the group cohomology of $\tilde{\Delta}$.

\begin{lemma}
\label{lemma:MV-isomorphism}
Let $\rho:\tilde{\Delta}\rightarrow \GL(V)$ be any $g$-dimensional complex representation with $V^\rho = \{0\}$. Let $A= \rho(\delta_0), B = \rho(\delta_1)$ and let $V^A,V^B$ be the eigenspaces of eigenvalue 1 of $A,B$, respectively. Then there is canonical isomorphism of complex vector spaces
$$
\frac{V}{V^A + V^B} \stackrel{\delta}\longrightarrow H^1(\tilde{\Delta},\rho)
$$
given as follows: write $v = v_1 - v_2 \in \frac{V}{V^A + V^B}$ and let $\delta(v) = \kappa_v$ be the 1-cocycle in $Z^1(\tilde{\Delta},\rho)$ determined by 
$$
\kappa_v(A) = \rho(A)v_1 - v_1, \quad \kappa_v(B) = \rho(B)v_2 - v_2.
$$
In particular, $\dim H^1(\tilde{\Delta},\rho) = g - d_0 - d_1$.
\end{lemma}

\begin{proof}

The group cohomology of a free product with two factors can be computed via the Mayer-Vietoris sequence, which in the case of $\tilde{\Delta} \simeq \ZZ/p\ZZ * \ZZ/q\ZZ $ gives the exact sequence of vector spaces
\begin{equation}
\label{equation:MVSeq}
\begin{aligned}
V^{\rho} &\rightarrow V^A\oplus V^B \xrightarrow{(v^A, v^B) \mapsto v^A - v^B} V \stackrel{\delta}\longrightarrow  H^1(\tilde{\Delta},\rho)&\rightarrow H^1(\ZZ/p\ZZ,\rho)\oplus H^1(\ZZ/q\ZZ,\rho),
 \end{aligned}
\end{equation}
where $\delta$ is the map given in the statement of the Lemma. We want to show that $\delta$ is an isomorphism. By assumption $V^{\rho} = \{0\}$. We also have that $H^1(\ZZ/N\ZZ,\rho) = 0$ for any $N\geq 1$. Indeed, if the group $\ZZ/N\ZZ$ is generated by $\gamma$, then any extension $\rho\rightarrow \pi \rightarrow 1$ of $\ZZ/N\ZZ$-representations is determined by the matrix 
$$
\pi(\gamma)  = \twobytwo{\rho(\gamma)}{\kappa(\gamma)}{0}{1}.
$$
But $\pi(\gamma)$ has finite order (dividing $N$) which means it is diagonalizable, and $\kappa(\gamma)$ differs from 0 by a coboundary. Therefore $H^1(\ZZ/p\ZZ,\rho) = H^1(\ZZ/q\ZZ,\rho) = 0$ and $\delta$ is an isomorphism. 
\end{proof}

Since the fixed vectors of $\rho_N$ are the differential forms on $X(\Delta)$, and since $X(\Delta)$ has genus zero, $\rho_N$ satisfies the hypothesis of Lemma \ref{lemma:MV-isomorphism}. Using this Lemma, we may define a canonical basis for $H^1(\tilde{\Delta},\rho_N)$ as follows: let $a = \dim V^A = d_0, b = \dim V^B = d_1$ and choose bases $\underline{v}^A=\{v_1^A, \ldots, v_a^A\}, \underline{v}^B= \{v_1^B, \ldots, v_b^B\}$ for $V^A, V^B$, respectively. Since $V^A\cap V^B =V^{\rho_N} = \{0\}$, we may extend $\underline{v}^A\cup \underline{v}^B$ to a basis $\{v_1, \ldots, v_d\}\cup\underline{v}^A\cup \underline{v}^B$ of $V$, where $d := g -a -b$. With respect to this choice of basis, the `fundamental cocycles' $\kappa_1, \ldots, \kappa_d$, determined by 
$$
\kappa_1(\delta_0) = \left( \begin{array}{c}
1 \\ 
0 \\ 
\vdots \\ 
0
\end{array} \right) = e_1, \quad \kappa_1(\delta_1) = \left( \begin{array}{c}
0 \\ 
\vdots \\ 
0
\end{array} \right), \quad \ldots, \quad \kappa_d(\delta_0) = \left( \begin{array}{c}
0 \\  
\vdots \\ 
0 \\
1 \\
0 \\
\vdots \\
0 
\end{array} \right) = e_d, \quad \kappa_d(\delta_1) = \left( \begin{array}{c}
0 \\ 
\vdots \\ 
0
\end{array} \right)
$$  
form a canonical basis for $H^1(\widetilde{\Delta},\rho_N)$. We next need a lemma which determines the `field of definition' of the $\kappa_i$'s. 

\begin{lemma}
\label{lemma:mainAlgebraicLemma}
Let $\rho:\widetilde{\Delta}\rightarrow \GL(V)$ be any finite-dimensional complex representation of finite image and with $V^\rho = \{0\}$. Suppose 
$
\dim H^1(\widetilde{\Delta},\rho) = d.
$
Then there is a choice of basis for $V$ such that the fundamental coycles $\kappa_i$ have entries in $\QQ(\zeta_e)$, where $\zeta_e$ is a primitive $e$-th root of unity and $e$ is the exponent of the finite group $\widetilde{\Delta}/\ker \rho$. 
\end{lemma}

\begin{proof}
Since $\rho$ is of finite image, a basis $\beta$ can be chosen so that the matrix entries of $\rho$ belong to $\QQ(\zeta_e)$, where $\zeta_e$ is a primitive $e$-th root of unity and $e$ is the exponent of the finite group $\widetilde{\Delta}/\ker \rho$ \cite{Brauer}. From this basis $\beta$ we may compute a basis for $V^A + V^B$ and extend it to a basis of $V$ to get fundamental cocyles $\kappa_i(A),\kappa_i(B)$, without changing the field of definition $\QQ(\zeta_e)$ of the entries of $\rho$, by basic linear algebra. Since $\kappa_i(A),\kappa_i(B) \in \ZZ^g$, the entries of a general vector $\kappa_i(\gamma)$, $\gamma \in \widetilde{\Delta}$, are polynomial in the matrix entries of $\rho$, thus they are also contained in  $\QQ(\zeta_e)$. 
\end{proof}

Let now $\pi_N$ be the representation of Theorem \ref{thm:MainVVmfThm}, viewed as a representation of $\widetilde{\Delta}\rightarrow \Delta$, and let $\kappa_{\pi_N}$ be $\Omega$, the 1-cocycle (\ref{eq:periods}) determined by $\pi_N$ as shown the proof of that theorem. Since $\rho_N$ has no fixed vectors, we can write $\Omega$ in terms of the fundamental cocycles $\kappa_i$, 
$$
\Omega = \sum_{i=1}^d \lambda_i\, \kappa_i , \quad \lambda_i \in \CC
$$
where $d = \dim H^1(\widetilde{\Delta}, \rho_N)$. As discussed after the proof of Theorem \ref{thm:MainVVmfThm}, the periods of the curve $X(N)$ are precisely the entries of $\Omega(\gamma)$, where $\gamma$ runs through the hyperbolic elements of $N$. The periods of $X(N)$ are thus the entries of the vectors
$$
\Omega(\gamma) = \sum_{i=1}^d \lambda_i\, \kappa_i(\gamma).
$$
Since $\rho_N$ has finite image, a basis can be chosen so that the entries of $\kappa_i(\gamma_i)$ are contained in $\QQ(\zeta_e)$, where $e$ is the exponent of the finite group $G=\Delta/N$, by Lemma \ref{lemma:mainAlgebraicLemma}. Therefore, the periods can be expressed as $\overline{\QQ}$-linear combinations of complex numbers  $\lambda_1, \ldots, \lambda_d \in \CC$ (which are potentially linearly dependent over $\overline{\QQ}$). By the result of Lemma \ref{lemma:MV-isomorphism}, we deduce that 
\begin{equation}
\label{eq:DeltaTildeBound}
\dim_{\overline{\QQ}} V_{X(N)} \leq d = \dim H^1(\widetilde{\Delta}, \rho_N) = g - d_0 - d_1
\end{equation}
which proves Theorem \ref{thm:MainThm} for the larger group $\widetilde{\Delta}$. In order to deduce the sharper result for $\Delta$, we use the fact that $\pi_N$ is actually a representation of $\Delta$, and therefore $\pi_N(\delta_{\infty}^{-r})  = I_{g+1}$. In particular, 
\begin{align*}
0 &= \Omega(\delta_{\infty}^{-r}) \\
&= \sum_{j=0}^{r-1} \rho_N(\delta^j_{\infty})\Omega(\delta_{\infty}^{-1}) \\
&= \sum_{j=0}^{r-1} \rho_N(\delta^j_{\infty})\Omega(\delta_0\delta_1) \\
&=  \sum_{j=0}^{r-1} \rho_N(\delta^j_{\infty})\Omega(\delta_0) \quad ( \text{since } \Omega(\delta_1) = \sum_{i=1}^d \lambda_i\, \kappa_i(\delta_1) = 0 ) \\
& = \sum_{i=1}^{d} \lambda_i T\kappa_i(\delta_0) 
\end{align*}
where for ease of notation we let $T:= \sum_{j=0}^{r-1} \rho_N(\delta^j_{\infty})$. Since $\rho_N$ is a representation of $\Delta$, $\rho_N(\delta_{\infty})$ is of order dividing $r$. It is therefore diagonalizable as  
$\rho_N(\delta_{\infty}) \sim \mathrm{diag}(\alpha_1, \ldots, \alpha_g)$, where each $\alpha_i$ is an $r$-th root of unity. By a standard computation with roots of unity it follows that $T$ has the block decomposition 
$$
T \sim \left( \begin{array}{cc} r\,I_{d_{\infty}} & 0 \\ 0 & 0  \end{array} \right) 
$$ 
with respect to a basis of eigenvectors for $\rho_N(\delta_{\infty})$, where the basis for the eigenspace of eigenvalue $\alpha = 1$ is placed first. Let $M$ be the change-of-basis matrix corresponding to this choice of basis. Then the entries of $M$ can be chosen to be in $\overline{\QQ}$, since the entries of $\rho_N(\delta_{\infty})$ are. We thus have the vector equation 
$$
0 = \sum_{i=1}^{d} \lambda_i \, M^{-1}\left( \begin{array}{cc} r\,I_{d_{\infty}} & 0 \\ 0 & 0  \end{array} \right)M \kappa_i(\delta_0) 
$$
which shows that we have an additional $d_{\infty}$-many independent $\overline{\QQ}$-linear relations among the $\lambda_i$'s. Combined with \eqref{eq:DeltaTildeBound}, we get the sharper bound
$$
\dim_{\overline{\QQ}} V_{X(N)} \leq  g - d_0 - d_1 - d_{\infty},
$$
which proves Theorem \ref{thm:MainThm}. 

\end{proof}

\begin{remark}
Consider the inflation-restriction exact sequence of vector spaces
$$
0 \rightarrow H^1(\Delta, \rho_N) \rightarrow H^1(\widetilde{\Delta}, \rho_N) \stackrel{\mathrm{Res}}\rightarrow H^1(K,\rho_N)^{\Delta} \rightarrow 0
$$
associated to the group homomorphisms $K\hookrightarrow \widetilde{\Delta} \rightarrow \Delta$, where $\mathrm{Res}(\kappa) = [\kappa(\delta_{\infty})^r]$. Then by the proof of Theorem \ref{thm:MainThm}  it is clear that 
$
\dim H^1(\Delta, \rho_N) = g - d_0 - d_1 - d_{\infty}
$
so the bound of Theorem \ref{thm:MainThm} can be stated more concisely as 
$$
\dim_{\overline{\QQ}} V_{X(N)} \leq  \dim H^1(\Delta, \rho_N).
$$
\end{remark}

\begin{remark}
The bound of Theorem \ref{thm:MainThm} tends to be {\em much} smaller than that of Theorem \ref{thm: WolfartBound}. Indeed, for $\Delta = \Delta(p,q,r)$ `on average' we expect that
\[\dim_{\CC} H^1(\Delta, \rho_N) \sim g - g/p - g/q - g/r   \ll g.\]
\end{remark}

The numbers $d_0, d_1, d_{\infty}$ are easy to compute. Indeed, the character $\chi_N$ of $\rho_N$ can be computed explicitly using the Chevalley-Weil formula  \cite{Chevalley-Weil} (see \cite{Candelori-CW} for a modern account), and the numbers $d_0, d_1, d_{\infty}$ can then be computed from the knowledge of $\chi_N$ and of the character table of the finite group $G = \Delta/N$. These computations can easily be implemented in GAP or MAGMA, as we show in the next section. 

\section{Examples}
\label{sec:Examples}

\begin{example}(Bolza surface) In genus two, the Riemann surface with largest automorphism group is known as the {\em Bolza surface}. This Riemann surface has many automorphisms, and it can be uniformized as $X(N)$, by an index 48 normal subgroup $N\triangleleft\, \Delta(2,3,8)$. The automorphism group is $G\simeq\GL(2,3)$. The canonical representation $\rho_N$ in this case is irreducible, which implies that  
$$
J(X) \sim A^k
$$
for a simple abelian variety $A$ \cite{Wolfart}. The only possibilities are either $\dim A = 1$, $k=2$ or $\dim A = 2, k=1$. We have 
$$
\dim H^1(\Delta(2,3,8), \rho_N) = 1
$$
and so by Theorems \ref{thm:MainThm} and \ref{thm:ShigaWolfartTheorem} 
$$
\dim_{\bar{\QQ}} V_X = \frac{2\dim A^2}{\dim_{\QQ} \End_0(A)} = 1.
$$
If $\dim A = 2$, then this can only happen if $\dim_{\QQ} \End_0(A) = 8$, which is impossible since $\dim_{\QQ} \End_0(A) \leq 4$ for any abelian surface $A$. Therefore $\Jac(X) \sim E^2$ for some elliptic curve $E$. Moreover, by Example \ref{ex:CMEcurveSchneider} we deduce that $E$ has complex multiplication, so that $X$ has complex multiplication as well. The period matrix of $X$ can be computed explicitly using our methods, as in \cite{Candelori-Marks}. 

\end{example}

\begin{example}[Klein quartic]There is a unique normal subgroup $N\triangleleft\,\Delta(2,3,7)=\Delta$ of index 168. The corresponding curve $X=X(N)$ is the famous `Klein quartic', of genus $g=3$. By Chevalley-Weil, the canonical representation $\rho_N$ of this curve is irreducible, and 
$$
\dim H^1(\Delta, \rho_N) = 1.
$$
Since $\rho_N$ is irreducible, this means that 
$
\Jac(X) \sim A^k
$,
where $A$ is a simple abelian variety \cite{Wolfart}. Since $\dim \Jac(X) = g = 3$,  either $\dim A = 1$ and $k=3$ or $\dim A = 3$ and $k=1$. We rule out the latter as follows. Using Thm. \ref{thm:ShigaWolfartTheorem} and Thm. \ref{thm:MainThm} we deduce that 
$$
\dim_{\bar{\QQ}} V_X = \frac{2\dim A^2}{\dim_{\QQ} \End_0(A)} = 1.
$$
If $\dim A = 3$, we get $\dim_{\QQ} \End_0(A)=18$. This is impossible, since for any abelian threefold, $\dim_{\QQ} \End_0(A) \leq  6$. Therefore $\dim A = 1, k=3$ and 
$
\Jac(X) \sim A^3
$
is totally decomposable and $\dim_{\QQ} \End_0(A) =2$, i.e., $A$ is a CM elliptic curve. Thus $J(X)$ is also CM, and again it has totally decomposable Jacobian.    
\end{example}

\begin{example}[Macbeath curve]There is a unique normal subgroup  $N\triangleleft\Delta(2,3,7) = \Delta$ of index 504. The corresponding curve $X = X(N)$ is the famous `Macbeath curve', of genus $g=7$. Again by Chevalley-Weil, this curve has an irreducible canonical representation $\rho_N$ with
$$
\dim H^1(\Delta, \rho_N) = 2.
$$
Since $\rho_N$ is irreducible, this means again that 
$
\Jac(X) \sim A^k,
$
with $A$ simple. Since $\dim \Jac(X) = g = 7$,  either $\dim A = 1$ and $k=7$ or $\dim A = 7$ and $k=1$.  Using Theorem \ref{thm:ShigaWolfartTheorem} and Theorem \ref{thm:MainThm} we deduce that 
$$\frac{\dim A^2}{\dim_{\QQ} \End_0(A)} \leq 1.$$ 
If $\dim A = 7$, we get $\dim_{\QQ} \End_0(A) \geq 49$, which is impossible. Therefore 
$
\Jac(X) \sim A^7
$
with $A=E$ an elliptic curve, so that $\Jac(X)$ is totally decomposable. The question of whether $E$ has complex multiplication or not was posed by Berry and Tretkoff in \cite{Berry-Tretkoff}. Later Wolfart \cite{Wolfart} proved that $E$ does not have CM by finding an explicit Weierstrass equation for $E$. Therefore 
\begin{equation}
\label{equation:dimH1BoundMacbeath}
\frac2{\dim A^2}{\dim_{\QQ} \End_0(A)} = 2 = \dim H^1(\Delta, \rho_N).
\end{equation}

\end{example}

\begin{example}
In genus 14, there are three Hurwitz curves with automorphism group isomorphic to $\PSL(2,13)$, of order 1092. They can all be uniformized by three distinct normal subgroups of index 1092 in $\Delta(2,3,7)$. The canonical representations are all irreducible, with 
$$
\dim H^1(X(N), \rho_N) = 2.
$$
Writing $\Jac(X) \sim A^k$, we must again have 
\begin{equation}
\label{equation:dimOfEnd2}
\frac{\dim A^2}{\dim_{\QQ} \End_0(A)} \leq 1,
\end{equation}
as for the Macbeath curve. This equality rules out $\dim A = 7, 14$ but both $\dim A = 1,2$ are possible. It is known in fact that $\Jac(X) \sim E^{14}$ for $A=E$ an elliptic curve, using the {\em group algebra decomposition} of $\Jac(X)$ \cite[13.6]{BL}, \cite{Paulhus}. By \eqref{equation:dimOfEnd2} we cannot conclude whether $E$ has complex multiplication or not. It would be interesting to determine whether or not each of the three genus 14 Hurwitz curves has complex multiplication or, equivalently, whether or not the equality \eqref{equation:dimH1BoundMacbeath} is attained for these Hurwitz curves. 
\end{example}

\begin{remark}
As shown in the above examples, it is possible to use our Theorem \ref{thm:MainThm} to give a sufficient criterion for the Riemann surface $X$ to have {\em complex multiplication}. This criterion is similar in spirit to that of Streit \cite{Streit}. It would be interesting to work out the exact relation between the two methods to detect complex multiplication.
\end{remark}

\begin{remark}
Using the group algebra decomposition of $\Jac(X)$ \cite[13.6]{BL} it is possible to give bounds for $\dim V_X$ via the Shiga-Wolfart Theorem \ref{thm:ShigaWolfartTheorem}. It would be interesting to determine precisely which part of the group algebra decomposition is determined by the bounds given in Theorem \ref{thm:MainThm}.
\end{remark}

\renewcommand\refname{References}
\bibliographystyle{alpha}
\bibliography{Jac}

\begin{thebibliography}{CHMY18}

\bibitem[Ban09]{Bantay}
P.~Bantay.
\newblock Vector-valued modular forms.
\newblock In {\em Vertex operator algebras and related areas}, volume 497 of
  {\em Contemp. Math.}, pages 19--31. Amer. Math. Soc., Providence, RI, 2009.

\bibitem[Bel79]{Belyi}
G.~V. Belyi.
\newblock Galois extensions of a maximal cyclotomic field.
\newblock {\em Izv. Akad. Nauk SSSR Ser. Mat.}, 43(2):267--276, 479, 1979.

\bibitem[BG07]{BG}
Peter Bantay and Terry Gannon.
\newblock Vector-valued modular functions for the modular group and the
  hypergeometric equation.
\newblock {\em Commun. Number Theory Phys.}, 1(4):651--680, 2007.

\bibitem[BL04]{BL}
Christina Birkenhake and Herbert Lange.
\newblock {\em Complex abelian varieties}, volume 302 of {\em Grundlehren der
  Mathematischen Wissenschaften [Fundamental Principles of Mathematical
  Sciences]}.
\newblock Springer-Verlag, Berlin, second edition, 2004.

\bibitem[Bra45]{Brauer}
Richard Brauer.
\newblock On the representation of a group of order {$g$} in the field of the
  {$g$}-th roots of unity.
\newblock {\em Amer. J. Math.}, 67:461--471, 1945.

\bibitem[BT92]{Berry-Tretkoff}
Kevin Berry and Marvin Tretkoff.
\newblock The period matrix of {M}acbeath's curve of genus seven.
\newblock In {\em Curves, {J}acobians, and abelian varieties ({A}mherst, {MA},
  1990)}, volume 136 of {\em Contemp. Math.}, pages 31--40. Amer. Math. Soc.,
  Providence, RI, 1992.

\bibitem[Can18]{Candelori-CW}
Luca Candelori.
\newblock The {C}hevalley-{W}eil formula for orbifold curves.
\newblock {\em SIGMA Symmetry Integrability Geom. Methods Appl.}, 14:Paper No.
  071, 17, 2018.

\bibitem[CHMY18]{Candelori-Marks}
Luca Candelori, Tucker Hartland, Christopher Marks, and Diego Y\'{e}pez.
\newblock Indecomposable vector-valued modular forms and periods of modular
  curves.
\newblock {\em Res. Number Theory}, 4(2):Art. 17, 24, 2018.

\bibitem[CW34]{Chevalley-Weil}
C.~Chevalley and A.~Weil.
\newblock \"{U}ber das verhalten der integrale 1. gattung bei automorphismen
  des funktionenk\"orpers.
\newblock {\em Abh. Math. Sem. Univ. Hamburg}, 10(1):358--361, 1934.

\bibitem[FGM18]{FGM}
Cameron Franc, Terry Gannon, and Geoffrey Mason.
\newblock On unbounded denominators and hypergeometric series.
\newblock {\em J. Number Theory}, 192:197--220, 2018.

\bibitem[FM14]{FM14}
Cameron Franc and Geoffrey Mason.
\newblock Fourier coefficients of vector-valued modular forms of dimension 2.
\newblock {\em Canad. Math. Bull.}, 57(3):485--494, 2014.

\bibitem[FM16]{FM16}
Cameron Franc and Geoffrey Mason.
\newblock Three-dimensional imprimitive representations of the modular group
  and their associated modular forms.
\newblock {\em J. Number Theory}, 160:186--214, 2016.

\bibitem[KM03]{KnoppMason1}
Marvin Knopp and Geoffrey Mason.
\newblock On vector-valued modular forms and their {F}ourier coefficients.
\newblock {\em Acta Arith.}, 110(2):117--124, 2003.

\bibitem[KM04]{KnoppMason2}
Marvin Knopp and Geoffrey Mason.
\newblock Vector-valued modular forms and {P}oincar\'e series.
\newblock {\em Illinois J. Math.}, 48(4):1345--1366, 2004.

\bibitem[Man72]{Manin}
Ju.~I. Manin.
\newblock Parabolic points and zeta functions of modular curves.
\newblock {\em Izv. Akad. Nauk SSSR Ser. Mat.}, 36:19--66, 1972.

\bibitem[Mar15]{Marks15}
Christopher Marks.
\newblock Fourier coefficients of three-dimensional vector-valued modular
  forms.
\newblock {\em Commun. Number Theory Phys.}, 9(2):387--412, 2015.

\bibitem[Mas07]{Mason07}
Geoffrey Mason.
\newblock Vector-valued modular forms and linear differential operators.
\newblock {\em Int. J. Number Theory}, 3(3):377--390, 2007.

\bibitem[Mas08]{Mason08}
Geoffrey Mason.
\newblock 2-dimensional vector-valued modular forms.
\newblock {\em Ramanujan J.}, 17(3):405--427, 2008.

\bibitem[Mas12]{Mason12}
Geoffrey Mason.
\newblock On the {F}ourier coefficients of 2-dimensional vector-valued modular
  forms.
\newblock {\em Proc. Amer. Math. Soc.}, 140(6):1921--1930, 2012.

\bibitem[Mil86]{Milne}
J.~S. Milne.
\newblock Jacobian varieties.
\newblock In {\em Arithmetic geometry ({S}torrs, {C}onn., 1984)}, pages
  167--212. Springer, New York, 1986.

\bibitem[MSSV02]{Magaard-Shaska}
K.~Magaard, T.~Shaska, S.~Shpectorov, and H.~V\"{o}lklein.
\newblock The locus of curves with prescribed automorphism group.
\newblock {\em RIMS Kyoto Series}, (1267):112--141, 2002.
\newblock Communications in arithmetic fundamental groups (Kyoto, 1999/2001).

\bibitem[Pau11]{Paulhus}
Jennifer Paulhus.
\newblock Jacobian decomposition of the hurwitz curve of genus 14.
\newblock {\em http://www.csc.villanova.edu/~jpaulhus/g14hurwitz.pdf}, 2011.

\bibitem[Shi94]{Shimura}
Goro Shimura.
\newblock {\em Introduction to the arithmetic theory of automorphic functions},
  volume~11 of {\em Publications of the Mathematical Society of Japan}.
\newblock Princeton University Press, Princeton, NJ, 1994.
\newblock Reprint of the 1971 original, Kan\^{o} Memorial Lectures, 1.

\bibitem[Sil86]{Silverman}
Joseph~H. Silverman.
\newblock {\em The arithmetic of elliptic curves}, volume 106 of {\em Graduate
  Texts in Mathematics}.
\newblock Springer-Verlag, New York, 1986.

\bibitem[Str01]{Streit}
M.~Streit.
\newblock Period matrices and representation theory.
\newblock {\em Abh. Math. Sem. Univ. Hamburg}, 71:279--290, 2001.

\bibitem[SW95]{Shiga-Wolfart}
Hironori Shiga and J\"{u}rgen Wolfart.
\newblock Criteria for complex multiplication and transcendence properties of
  automorphic functions.
\newblock {\em J. Reine Angew. Math.}, 463:1--25, 1995.

\bibitem[Wol02]{Wolfart}
J\"{u}rgen Wolfart.
\newblock Regular dessins, endomorphisms of {J}acobians, and transcendence.
\newblock In {\em A panorama of number theory or the view from {B}aker's garden
  ({Z}\"{u}rich, 1999)}, pages 107--120. Cambridge Univ. Press, Cambridge,
  2002.

\bibitem[Zhu96]{Zhu}
Yongchang Zhu.
\newblock Modular invariance of characters of vertex operator algebras.
\newblock {\em J. Amer. Math. Soc.}, 9(1):237--302, 1996.

\end{thebibliography}
\end{document}